\documentclass{amiart}

\newtheorem{proposition}[theorem]{Proposition}
\definecolor{CORbleu}{RGB}{0,0,0}
\newcommand{\BLU}[1]{\textcolor{CORbleu}{#1}}

\title{Generalized Steinhaus triangles generated by canonical basis vectors:
periodicity and weight formulas}
\author{Randa Ouchene\inst1, Hac\`ene Belbachir\inst1}
\institute{
  \inst1 \BLU{USTHB, Faculty of Mathematics, RECITS Laboratory, Algiers, Algeria}\\
  \url{randa.ouchene1@uqac.ca}, \url{randaaouchene@gmail.com},
  \url{hbelbachir@usthb.dz},
  \url{hacenebelbachir@gmail.com}
}

\begin{document}
\maketitle

\begin{abstract}
A generalized Steinhaus $s$-triangle is obtained from a binary sequence by
repeatedly replacing each block of $s$ consecutive entries by its sum modulo
$2$. We study the triangles generated by canonical basis vectors. Expressing
their entries through bi$^{s}$nomial coefficients, we prove that the truncated
coefficient profiles of the successive rows are purely periodic and determine
their exact period. This periodicity yields a power-of-two derivative identity
and a decomposition into identical finite blocks. We consequently obtain an
exact recurrence for the number of ones and show that, on every residue class,
this weight is an affine function of the length. We also derive a rational
generating function, determine the asymptotic growth rate, and prove that these
canonical triangles have zero density of ones. Explicit formulas are obtained
for the first two canonical vectors and for the block weight in the classical
case. These results provide a unified description of the periodic and
enumerative structure of generalized Steinhaus triangles generated by
canonical basis vectors.
\end{abstract}


\section{Introduction}

Steinhaus triangles are generated from a finite binary sequence by repeated
adjacent addition modulo $2$. The classical balanced-triangle problem goes
back to Steinhaus and Harborth~\cite{4,3}. Several generalizations based on
different recurrence rules have since been considered, including generalized
Pascal addition~\cite{6} and Delannoy-type recurrences~\cite{ouchene2026delannoy}.
In the present study, we focus on the generalized Pascal rule, in which each
new entry is the sum modulo $2$ of $s$ consecutive entries. This construction
connects Steinhaus triangles with generalized Pascal triangles and with the
coefficients of powers of $1+x+\cdots+x^{s-1}$.

Several aspects of these structures are already understood. Brunat and
Maureso studied symmetries in ordinary and generalized Pascal
triangles~\cite{brunat2011}, while Harborth and Hurlbert investigated the
number of ones in general binary Pascal triangles~\cite{harborth2005}. For
the classical rule $s=2$, Brunat and Maureso also analyzed triangles
generated by canonical basis vectors~\cite{1}. On the algebraic side,
congruence, recurrence, and combinatorial properties of bi$^{s}$nomial
coefficients have been developed by Belbachir and
Igueroufa~\cite{belbachir2020a,belbachir2020b} and Belbachir and
Otmani~\cite{belbachir2024}. These works provide the tools required to study
canonical generators under generalized Pascal addition, but the exact row
periodicity and its enumerative consequences have not been isolated for this
family.

Our contribution is to fill this gap. We express the entries of a generalized
Steinhaus triangle in terms of bi$^{s}$nomial coefficients and encode the
relevant part of each row by a truncated coefficient profile. Repeated
squaring over $\mathbb F_2$ gives the exact period of these profiles. We use
this result to derive the canonical-vector derivative
identity, a repeated-block decomposition, and an exact affine formula on each
residue class of the length. The same framework yields a rational generating
function, asymptotic sparsity, closed formulas for the first two canonical
vectors, and an explicit specialization for the ordinary Pascal rule.

The remainder of the paper is organized as follows.
Section~\ref{sec:generalized-triangles} introduces generalized Steinhaus
triangles, bi$^{s}$nomial coefficients, and the entry formula used throughout
the paper. Section~\ref{sec:canonical-vectors} specializes this formula to
canonical basis vectors and establishes the first explicit weight formulas.
Section~\ref{sec:profile-periodicity} introduces the truncated row profiles
and determines their exact power-of-two period by repeated squaring over
$\mathbb F_2$. Section~\ref{sec:block-weights} uses this periodicity to obtain
the block decomposition, the weight recurrence, the affine residue-class
formula, the asymptotic density, the rational generating function, and the
classical specialization.
Finally, Section~\ref{sec:conclusion} summarizes the results and presents
directions for further research.

\section{Generalized Steinhaus triangles}
\label{sec:generalized-triangles}

Let $X=(a_0,a_1,\ldots,a_{n-1})$ be a binary sequence and let
$2\leq s\leq n$. Its $s$-derivative is
\[
\partial_sX=(a'_0,a'_1,\ldots,a'_{n-s}),
\qquad
a'_i=a_i+\cdots+a_{i+s-1}\pmod2.
\]
Thus $\partial_sX$ has length $n-s+1$. Put
\[
\ell=\left\lfloor\frac{n-1}{s-1}\right\rfloor.
\]
Starting with $\partial_s^0X=X$, define
$\partial_s^rX=\partial_s(\partial_s^{r-1}X)$ whenever the latter
derivative exists. The last row $\partial_s^\ell X$ has positive length
smaller than $s$. We denote by
\[
T^s(X)=(X,\partial_sX,\ldots,\partial_s^\ell X)
\]
the generalized Steinhaus $s$-triangle associated with $X$.

\begin{example}[A first generalized triangle]\label{ex:first-generalized}
Take $s=3$ and
\[
X=(0,0,1,0,0,0,0).
\]
Each entry of the next row is the sum modulo $2$ of a block of three
consecutive entries. For instance,
\[
(\partial_3X)_0=0+0+1=1,
\qquad
(\partial_3^2X)_1=1+1+0=0.
\]
The complete triangle is
\[
\begin{array}{c@{\qquad}ccccccc}
r=0&0&0&1&0&0&0&0\\
r=1&&1&1&1&0&0&\\
r=2&&&1&0&1&&\\
r=3&&&&0&&&
\end{array}
\]
and therefore
\[
|T^3(X)|=1+3+2+0=6.
\]
Unlike the classical case $s=2$, the length decreases here by
$s-1=2$ from one row to the next. The generating sequence in this
example is the canonical basis vector $e_2^{(7)}$, in the notation
introduced below.
\end{example}

The weight $|X|$ is the number of ones in $X$, and the weight of the
triangle is
\[
|T^s(X)|=\sum_{r=0}^{\ell}|\partial_s^rX|.
\]

\subsection{Bi$^{s}$nomial coefficients}

For nonnegative integers $r$ and $j$, the bi$^{s}$nomial coefficient
$\binom rj_s$ is defined by
\[
(1+x+\cdots+x^s)^r
=\sum_{j=0}^{rs}\binom rj_s x^j,
\]
with $\binom rj_s=0$ when $j<0$ or $j>rs$. \BLU{Thus the subscript always
indicates the degree of the underlying polynomial: for the $s$-term rule
studied in this paper, the entries of the triangles are the coefficients
$\binom{r}{j}_{s-1}$ of $(1+x+\cdots+x^{s-1})^{r}$.} These coefficients satisfy
the symmetry relation
\[
\binom rj_s=\binom r{rs-j}_s
\]
and the generalized Pascal recurrence
\[
\binom rj_s
=\binom{r-1}{j}_s+\binom{r-1}{j-1}_s+\cdots
+\binom{r-1}{j-s}_s.
\]

\begin{proposition}\label{prop:entry}
For every valid position $(r,c)$ of $T^s(X)$,
\begin{equation}\label{eq:general-entry}
T^s(X)(r,c)
=
\sum_{i=0}^{r(s-1)}
\binom ri_{s-1}a_{c+i}
\pmod2.
\end{equation}
\end{proposition}

\begin{proof}
The assertion is immediate for $r=0$. Assume that it holds in row
$r$. By the definition of the $s$-derivative,
\[
T^s(X)(r+1,c)=\sum_{q=0}^{s-1}T^s(X)(r,c+q).
\]
Substituting the induction hypothesis and collecting the coefficient
of each $a_{c+i}$ gives
\[
T^s(X)(r+1,c)
=
\sum_{i=0}^{(r+1)(s-1)}
\left(
\sum_{q=0}^{s-1}\binom r{i-q}_{s-1}
\right)a_{c+i}
\pmod2.
\]
The inner sum equals $\binom{r+1}{i}_{s-1}$ by the generalized
Pascal recurrence.
\end{proof}

\section{Canonical basis vectors}
\label{sec:canonical-vectors}

Let $n\geq1$ and $0\leq k\leq n-1$. The $k$-th canonical basis
vector is
\[
e_k^{(n)}=(\underbrace{0,\ldots,0}_{k},1,
\underbrace{0,\ldots,0}_{n-k-1}).
\]
We write
\[
T^s(k,n)=T^s(e_k^{(n)})
\qquad\text{and}\qquad
w_s(k,n)=|T^s(k,n)|.
\]
Applying Proposition~\ref{prop:entry} to $e_k^{(n)}$ gives
\begin{equation}\label{eq:canonical-entry}
T^s(k,n)(r,c)=
\binom r{k-c}_{s-1}\pmod2
\end{equation}
for every valid position $(r,c)$. In particular, all entries in
columns $c>k$ vanish.

The symmetry of the bi$^{s}$nomial coefficients gives
\[
T^s(n-1-k,n)(r,n-1-r(s-1)-c)=T^s(k,n)(r,c).
\]
Consequently,
\[
w_s(k,n)=w_s(n-1-k,n).
\]

\begin{proposition}\label{prop:first-vectors}
Let $a=s-1$.
\begin{enumerate}
\item For $n\geq s$,
\[
w_s(0,n)=\left\lceil\frac{n}{a}\right\rceil.
\]
\item For $n\geq s$, put $L=\lfloor(n-1)/a\rfloor$. Then
\[
w_s(1,n)=
\left\lceil\frac{n-1}{a}\right\rceil
+\left\lfloor\frac{L+1}{2}\right\rfloor.
\]
\end{enumerate}
\end{proposition}
\begin{proof}
Set \(a=s-1\).

For \(k=0\), Equation~\eqref{eq:canonical-entry} gives
\[
\left(\partial_s^r e_0^{(n)}\right)_0
=
\binom{r}{0}_{s-1}
=
\binom{r}{0}
=
1.
\]
Thus every nonempty row has exactly one nonzero entry, located in
column \(0\). Since row \(r\) has length \(n-ar\), the admissible row
indices are
\[
0\leq r\leq \left\lfloor\frac{n-1}{a}\right\rfloor.
\]
Therefore, the number of nonempty rows is
\[
\left\lfloor\frac{n-1}{a}\right\rfloor+1
=
\left\lceil\frac{n}{a}\right\rceil,
\]
and hence
\[
w_s(0,n)=\left\lceil\frac{n}{a}\right\rceil.
\]

Now let \(k=1\). By Equation~\eqref{eq:canonical-entry}, the only
columns that may contain a nonzero entry are columns \(0\) and \(1\).

In column \(1\), we have
\[
\left(\partial_s^r e_1^{(n)}\right)_1
=
\binom{r}{0}_{s-1}
=
\binom{r}{0}
=
1.
\]
Hence column \(1\) contributes one in every row having length at least
\(2\). The number of such rows is
\[
\left\lceil\frac{n-1}{a}\right\rceil.
\]

In column \(0\), we have
\[
\left(\partial_s^r e_1^{(n)}\right)_0
=
\binom{r}{1}_{s-1}
=
\binom{r}{1}
=
r
\pmod 2.
\]
Indeed, to obtain total degree \(1\), one chooses the term \(x\) from
exactly one of the \(r\) factors and the constant term \(1\) from all
the others. There are therefore \(\binom r1=r\) such choices.

Consequently, the entry in column \(0\) is equal to \(1\) precisely
when \(r\) is odd. If
\[
L=\left\lfloor\frac{n-1}{a}\right\rfloor,
\]
then the nonempty rows are indexed by \(r=0,\ldots,L\). The number of
odd integers in this interval is
\[
\left\lfloor\frac{L+1}{2}\right\rfloor.
\]

Adding the contributions of columns \(1\) and \(0\), we obtain
\[
w_s(1,n)
=
\left\lceil\frac{n-1}{a}\right\rceil
+
\left\lfloor\frac{L+1}{2}\right\rfloor.
\]
\end{proof}

\section{Periodicity and recurrence of canonical-basis rows}
\label{sec:profile-periodicity}

Fix $s\geq2$ and $k\geq1$. For $r\geq0$, define the truncated
coefficient profile
\[
\mathbf v_r^{(s,k)}
=
\left(
\binom r0_{s-1},
\binom r1_{s-1},
\ldots,
\binom rk_{s-1}
\right)\pmod2
\]
and let $\rho_{s,k}(r)$ denote its Hamming weight. Whenever the
$r$-th row of $T^s(k,n)$ contains all columns from $0$ to $k$,
Equation~\eqref{eq:canonical-entry} shows that its weight is
$\rho_{s,k}(r)$.

\begin{theorem}\label{thm:profile-period}
Let $h=2^t$ be the least power of two strictly greater than $k$; thus
$2^{t-1}\leq k<2^t$. The sequence
\[
\left(\mathbf v_r^{(s,k)}\right)_{r\geq0}
\]
is purely periodic, and its least period is $h$.
\end{theorem}

\begin{proof}
Since $h=2^t$, repeated squaring modulo $2$ yields
\[
\begin{aligned}
\binom h0_{s-1}&\equiv1\pmod2,
&\binom hh_{s-1}&\equiv1\pmod2,\\
\binom h\ell_{s-1}&\equiv0\pmod2
&&\text{for }1\leq \ell<h.
\end{aligned}
\]
For $0\leq j\leq k<h$, the generalized Vandermonde convolution for
polynomial coefficients~\cite[Table~1, Vandermonde convolution]{Fahssi2012} gives
\[
\binom{r+h}{j}_{s-1}
=
\sum_{\ell=0}^{j}
\binom r{j-\ell}_{s-1}
\binom h\ell_{s-1}
\equiv
\binom rj_{s-1}
\pmod2.
\]
Thus $\mathbf v_{r+h}^{(s,k)}=\mathbf v_r^{(s,k)}$, and $h$ is a
period.

It remains to prove minimality. Suppose that $d$ is a period. Since
$\mathbf v_d^{(s,k)}=\mathbf v_0^{(s,k)}$, we have
\[
\binom dj_{s-1}=0\pmod2
\qquad(1\leq j\leq k).
\]
Write $d=qm$, where $q=2^u$ and $m$ is odd. If $q\leq k$, repeated
squaring modulo $2$ gives
\[
\binom q0_{s-1}\equiv\binom qq_{s-1}\equiv1\pmod2,
\qquad
\binom q\ell_{s-1}\equiv0\pmod2\quad(1\leq\ell<q).
\]
Applying the generalized Vandermonde convolution to the sum of $m$
copies of $q$, the only contributions that are nonzero modulo $2$ in
$\binom dq_{s-1}$ are obtained by
choosing degree $q$ from one copy and degree $0$ from all the others.
Hence
\[
\binom dq_{s-1}\equiv m\equiv1\pmod2,
\]
contradicting $q\leq k$. Therefore $q>k$, so $q\geq h$ and $h\mid d$.
Thus no smaller positive period exists.
\end{proof}

\begin{corollary}\label{cor:derivative}
Let $P=(s-1)h$. If $n\geq k+1+P$, then
\[
\partial_s^h e_k^{(n)}=e_k^{(n-P)}.
\]
\end{corollary}

\begin{proof}
By Theorem~\ref{thm:profile-period},
$\mathbf v_h^{(s,k)}=\mathbf v_0^{(s,k)}=(1,0,\ldots,0)$.
Equation~\eqref{eq:canonical-entry} then shows that row $h$ has its
unique nonzero entry in column $k$. Its length is $n-P\geq k+1$.
\end{proof}

\section{Block decomposition and weight formulas}
\label{sec:block-weights}

With the notation of Theorem~\ref{thm:profile-period}, define
\begin{equation}\label{eq:block-weight}
S_{s,k}=\sum_{r=0}^{h-1}\rho_{s,k}(r)
=
\sum_{r=0}^{h-1}\sum_{j=0}^{k}
\left(\binom rj_{s-1}\bmod2\right).
\end{equation}
Thus $S_{s,k}$ is computable from a finite coefficient block that
depends only on $s$ and $k$, not on $n$.

\BLU{The following table lists $S_{s,k}$ for $2\leq s\leq5$ and $1\leq k\leq8$.
\begin{center}
\begin{tabular}{c|cccccccc}
$s\backslash k$&1&2&3&4&5&6&7&8\\\hline
$2$&3&8&9&22&24&26&27&62\\
$3$&3&8&9&22&24&28&29&66\\
$4$&3&8&9&22&24&26&27&62\\
$5$&3&8&9&22&24&28&29&66
\end{tabular}
\end{center}
Note the coincidence $S_{4,k}=S_{2,k}$, which follows from the factorization
$1+x+x^{2}+x^{3}=(1+x)^{3}$ over $\mathbb F_2$ together with the fact that
multiplication by an odd constant permutes the residues modulo $2^{t}$, so
that $\binom{3r}{j}\bmod2$ and $\binom{r}{j}\bmod2$ have the same column
statistics over a full period.}

\begin{theorem}\label{thm:weight-formula}
Let $k\geq1$, let $h$ be the least power of two greater than $k$, and
put $P=(s-1)h$. Then
\begin{equation}\label{eq:weight-recurrence}
w_s(k,n)=w_s(k,n-P)+S_{s,k}
\end{equation}
for every $n\geq k+1+P$.

More precisely, for every $n\geq P$, write $n=qP+r$, where $q\geq1$
and $0\leq r<P$, and set $B_r=w_s(k,P+r)$. Then
\begin{equation}\label{eq:quasipolynomial}
w_s(k,n)=B_r+(q-1)S_{s,k}.
\end{equation}
\end{theorem}

\begin{proof}
By Corollary~\ref{cor:derivative}, deleting the first \(h\) rows of
\(T^s(k,n)\) leaves \(T^s(k,n-P)\). Since \(n-P\geq k+1\), each
deleted row contains columns \(0,\ldots,k\); its weight is therefore
\(\rho_{s,k}(i)\) in row \(i\). Hence the total weight of the deleted
rows is \(S_{s,k}\), which proves~\eqref{eq:weight-recurrence}.

Now write \(n=qP+r\), where \(q\geq1\) and \(0\leq r<P\). Since
\(P=(s-1)h\geq h>k\), the recurrence may be iterated until the
argument is \(P+r\). Thus
\[
w_s(k,n)=w_s(k,P+r)+(q-1)S_{s,k}
=B_r+(q-1)S_{s,k},
\]
proving~\eqref{eq:quasipolynomial}.
\end{proof}

\begin{corollary}\label{cor:asymptotic-density}
For fixed \(s\geq2\) and \(k\geq1\),
\[
\lim_{n\to\infty}\frac{w_s(k,n)}{n}
=\frac{S_{s,k}}{P}.
\]
Moreover, if
\[
N_s(n)=
\sum_{i=0}^{\lfloor(n-1)/(s-1)\rfloor}
\bigl(n-i(s-1)\bigr)
\]
denotes the total number of entries in a generalized \(s\)-triangle
of order \(n\), then
\[
\lim_{n\to\infty}\frac{w_s(k,n)}{N_s(n)}=0.
\]
The same zero-density conclusion holds for \(k=0\).
\end{corollary}

\begin{proof}
Let
\[
n=qP+r,
\qquad q\geq 1,
\qquad 0\leq r<P.
\]
By~\eqref{eq:quasipolynomial}, we have
\[
w_s(k,n)=B_r+(q-1)S_{s,k}.
\]
Since \(q=(n-r)/P\), substituting this expression for \(q\) gives
\[
\begin{aligned}
w_s(k,n)
&=B_r+\left(\frac{n-r}{P}-1\right)S_{s,k}\\
&=\frac{S_{s,k}}{P}\,n+
\left(
B_r-S_{s,k}-\frac{S_{s,k}}{P}\,r
\right).
\end{aligned}
\]

The expression in parentheses depends only on the remainder \(r\).
Since \(0\leq r<P\), the integer \(r\) can take only the finitely many
values
\[
0,1,\ldots,P-1.
\]
Consequently, \(B_r\) also ranges over only finitely many values.
Therefore, there exists a constant \(C_{s,k}>0\), independent of \(n\),
such that
\[
\left|
B_r-S_{s,k}-\frac{S_{s,k}}{P}\,r
\right|
\leq C_{s,k}.
\]
Hence,
\[
w_s(k,n)=\frac{S_{s,k}}{P}\,n+O(1).
\]
Dividing by \(n\), we obtain
\[
\frac{w_s(k,n)}{n}
=
\frac{S_{s,k}}{P}
+
O\left(\frac1n\right).
\]
Since \(O(1/n)\) tends to zero as \(n\to\infty\), it follows that
\[
\lim_{n\to\infty}\frac{w_s(k,n)}{n}
=
\frac{S_{s,k}}{P}.
\]

We now consider the total number of entries in the generalized
Steinhaus triangle. Set
\[
a=s-1
\qquad\text{and}\qquad
\ell=\left\lfloor\frac{n-1}{a}\right\rfloor.
\]
The successive row lengths are
\[
n,\ n-a,\ n-2a,\ \ldots,\ n-\ell a.
\]
Thus, \(N_s(n)\) is the sum of these \(\ell+1\) row lengths:
\[
\begin{aligned}
N_s(n)
&=\sum_{i=0}^{\ell}(n-ia)\\
&=(\ell+1)n-a\sum_{i=0}^{\ell}i\\
&=(\ell+1)n-\frac{a\ell(\ell+1)}{2}.
\end{aligned}
\]
Moreover,
\[
\ell=\frac{n}{a}+O(1),
\]
because taking the integer part changes \((n-1)/a\) by less than one.
Substituting this estimate into the preceding expression yields
\[
N_s(n)=\frac{n^2}{2a}+O(n).
\]
Therefore, the total number of entries grows quadratically with \(n\),
whereas
\[
w_s(k,n)=\frac{S_{s,k}}{P}\,n+O(1)
\]
grows only linearly. Consequently,
\[
\frac{w_s(k,n)}{N_s(n)}
=
O\left(\frac1n\right),
\]
and hence
\[
\lim_{n\to\infty}\frac{w_s(k,n)}{N_s(n)}=0.
\]

Finally, when \(k=0\), we have
\[
w_s(0,n)=\left\lceil\frac{n}{a}\right\rceil
=\frac{n}{a}+O(1).
\]
Thus \(w_s(0,n)=O(n)\), while \(N_s(n)\) remains of order \(n^2\).
The same argument therefore gives
\[
\lim_{n\to\infty}\frac{w_s(0,n)}{N_s(n)}=0.
\]
\end{proof}
\begin{corollary}\label{cor:generating-function}
Let
\[
B_{s,k}(z)=\sum_{r=0}^{P-1}B_rz^r
\]
and
\[
\mathcal W_{s,k}(z)
=
\sum_{n\geq P}w_s(k,n)z^n.
\]
Then
\[
\mathcal W_{s,k}(z)
=
\frac{z^P B_{s,k}(z)}{1-z^P}
+
\frac{S_{s,k}z^{2P}}
{(1-z)(1-z^P)}.
\]
In particular, the ordinary generating function of the weight sequence
is rational.
\end{corollary}

\begin{proof}
By~\eqref{eq:quasipolynomial},
\[
\begin{aligned}
\mathcal W_{s,k}(z)
&=
\sum_{r=0}^{P-1}\sum_{q\geq1}
\bigl(B_r+(q-1)S_{s,k}\bigr)z^{qP+r}\\
&=
\left(\sum_{r=0}^{P-1}B_rz^r\right)
\left(\sum_{q\geq1}z^{qP}\right)\\
&\quad+
S_{s,k}
\left(\sum_{r=0}^{P-1}z^r\right)
\left(\sum_{q\geq1}(q-1)z^{qP}\right).
\end{aligned}
\]
Using
\[
\sum_{q\geq1}z^{qP}
=
\frac{z^P}{1-z^P},
\qquad
\sum_{q\geq1}(q-1)z^{qP}
=
\frac{z^{2P}}{(1-z^P)^2},
\]
and the finite geometric sum
\[
\sum_{r=0}^{P-1}z^r
=
1+z+\cdots+z^{P-1}
=
\frac{1-z^P}{1-z},
\]
we obtain
\[
\begin{aligned}
\mathcal W_{s,k}(z)
&=
B_{s,k}(z)\frac{z^P}{1-z^P}
+
S_{s,k}\frac{1-z^P}{1-z}
\frac{z^{2P}}{(1-z^P)^2}\\
&=
\frac{z^P B_{s,k}(z)}{1-z^P}
+
\frac{S_{s,k}z^{2P}}
{(1-z)(1-z^P)}.
\end{aligned}
\]
The right-hand side is a sum of quotients of polynomials and is
therefore rational.
\end{proof}
\subsection{The classical block weight}

\begin{corollary}\label{cor:classical-block}
For the classical rule \(s=2\),
\[
S_{2,k}
=
\sum_{j=0}^{k}2^{t-\operatorname{pop}_2(j)},
\]
where \(h=2^t>k\) and \(\operatorname{pop}_2(j)\) denotes the number of
ones in the binary expansion of \(j\). In particular, if
\(k=2^t-1\), then
\[
S_{2,k}=3^t.
\]
\end{corollary}

\begin{proof}
For \(s=2\), the generalized binomial coefficient appearing in
\eqref{eq:block-weight} reduces to the ordinary binomial coefficient:
\[
\binom{r}{j}_{1}=\binom{r}{j}.
\]
Therefore,
\[
S_{2,k}
=
\sum_{r=0}^{h-1}
\sum_{j=0}^{k}
\left(\binom{r}{j}\bmod 2\right).
\]
Since both sums are finite, their order may be exchanged:
\[
S_{2,k}
=
\sum_{j=0}^{k}
\#\left\{
0\leq r<h:
\binom{r}{j}\equiv1\pmod2
\right\}.
\]

Write the binary expansions of \(j\) and \(r\) as
\[
j=\sum_{\ell=0}^{t-1}j_\ell2^\ell,
\qquad
r=\sum_{\ell=0}^{t-1}r_\ell2^\ell,
\]
where \(j_\ell,r_\ell\in\{0,1\}\). By Lucas' theorem,
\[
\binom{r}{j}
\equiv
\prod_{\ell=0}^{t-1}
\binom{r_\ell}{j_\ell}
\pmod2.
\]
Consequently, \(\binom{r}{j}\) is odd if and only if
\[
j_\ell\leq r_\ell
\qquad
\text{for every }0\leq\ell<t.
\]
Equivalently, every binary position occupied by a \(1\) in \(j\)
must also be occupied by a \(1\) in \(r\).

If \(\operatorname{pop}_2(j)\) is the number of ones in the binary
expansion of \(j\), then these \(\operatorname{pop}_2(j)\) digits of
\(r\) are forced to equal \(1\). Each of the remaining
\(t-\operatorname{pop}_2(j)\) digits of \(r\) may independently be
chosen to be \(0\) or \(1\). Hence
\[
\#\left\{
0\leq r<h:
\binom{r}{j}\equiv1\pmod2
\right\}
=
2^{t-\operatorname{pop}_2(j)}.
\]
Substitution gives
\[
S_{2,k}
=
\sum_{j=0}^{k}
2^{t-\operatorname{pop}_2(j)}.
\]

Suppose now that \(k=2^t-1\). Then \(j\) runs through all binary
strings of length \(t\). For each binary position, a digit equal to
\(0\) in \(j\) contributes a factor \(2\), whereas a digit equal to
\(1\) contributes a factor \(1\). Therefore, the product rule over
the \(t\) independent binary positions gives
\[
\sum_{j=0}^{2^t-1}
2^{t-\operatorname{pop}_2(j)}
=
(2+1)^t
=
3^t.
\]
Thus
\[
S_{2,2^t-1}=3^t.
\]
\end{proof}

\subsection{Example: Weight computation for a large triangle}

Consider again the parameters $s=3$ and $k=2$ from
Example~\ref{ex:first-generalized}, now with $n=2019$. Here,
$h=4$ and $P=(s-1)h=8$. One complete block is
\[
\underbrace{
\begin{array}{ccc}
1&0&0\\
1&1&1\\
1&0&1\\
1&1&0
\end{array}}_{\mathbf v_0^{(3,2)},\ldots,\mathbf v_3^{(3,2)}}
\qquad\Longrightarrow\qquad
S_{3,2}=1+3+2+2=8.
\]

Moreover,
\[
2019=252\cdot8+3,
\]
so the required initial value is
\[
B_3=w_3(2,11).
\]
It is obtained from the following triangle:
\[
\underbrace{
\begin{array}{ccccccccccc}
0&0&1&0&0&0&0&0&0&0&0\\
&1&1&1&0&0&0&0&0&0&\\
&&1&0&1&0&0&0&0&&\\
&&&0&1&1&0&0&&&\\
&&&&0&0&1&&&&\\
&&&&&1&&&&&
\end{array}}_{T^3(e_2^{(11)})}
\qquad\Longrightarrow\qquad
B_3=1+3+2+2+1+1=10.
\]

Therefore, Theorem~\ref{thm:weight-formula} yields
\[
w_3(2,2019)
=
B_3+(252-1)S_{3,2}
=
10+251\cdot8
=
2018.
\]
\section{Conclusion}
\label{sec:conclusion}

We studied generalized Steinhaus triangles generated by canonical
basis vectors under $s$-term Pascal addition. The bi$^{s}$nomial
entry formula led to a truncated coefficient profile for each row.
Repeated squaring over $\mathbb F_2$ shows that these profiles are
purely periodic and determines their exact power-of-two period. This
periodicity yields the canonical-vector derivative identity.

The resulting block decomposition yields an exact recurrence for the
weight, an affine formula on every residue class of the length, and a
rational generating function. It also determines the asymptotic
growth rate and shows that triangles generated from a fixed canonical
position are sparse. The explicit formulas for $k=0$ and $k=1$, the
classical block formula, and the numerical example illustrate how the
general framework produces more specialized enumerative results.

Several questions remain open. The most immediate is to evaluate
$S_{s,k}$ explicitly from the binary expansions of $s$ and $k$, or
to obtain a finite automaton that computes it without constructing
the complete block. It would also be interesting to extend the method
from canonical vectors to sparse or structured generating sequences.
Finally, a graph-theoretic extension requires a fully
specified square-matrix construction, such as the generalized
Steinhaus graphs of Brand and Morton~\cite{brand1995,brand1996}; this
provides a natural direction for a separate study.

\sloppy
\section*{Funding}
The authors declare that no funds, grants, or other support were received during the preparation of this manuscript.

\section*{Competing interests}
The authors have no relevant financial or non-financial interests to disclose.

\section*{Declaration on the use of generative AI}
All mathematical content of this article is the authors’ own work. Generative AI (Claude, Anthropic) was used solely for language editing and formatting, under the authors’ full review and responsibility.

{\footnotesize\bibliographystyle{abbrv}
\bibliography{bib2}}

\end{document}